\documentclass{article}

\usepackage{amsmath, amssymb, amsthm}
\usepackage{pstricks}
\usepackage{authblk}

\def\margin{2.30cm}
\usepackage[left=\margin,right=\margin,top=\margin,bottom=\margin]{geometry}
\title{A Dirac-type Characterization of $k$-chordal Graphs}
\author[1]{R. Krithika}
\author[2]{Rogers Mathew}
  \author[1]{N. S. Narayanaswamy}
\author[1]{N. Sadagopan}
\affil[1]{
	Department of Computer Science and Engineering, \authorcr
	Indian Institute of Technology Madras, India. \authorcr
	\{krithika, swamy, sadagopu\}@cse.iitm.ac.in
}
\affil[2]{
	Department of Computer Science and Automation, \authorcr 
	Indian Institute of Science, India. \authorcr
	rogers@csa.iisc.ernet.in
}

\date{}

\theoremstyle{plain}
\newtheorem{theorem}{Theorem}
\newtheorem{lemma}[theorem]{Lemma}

\theoremstyle{remark}
\newtheorem{obs}[theorem]{Observation}

\begin{document}
\maketitle
\begin{abstract}
Characterization of $k$-chordal graphs based on the existence of a {\em simplicial path} was shown in [Chv{\'a}tal et al. Note: Dirac-type characterizations of graphs without long chordless cycles. \emph{Discrete Mathematics, 256, 445-448}, 2002]. We give a characterization of $k$-chordal graphs which is a generalization of the known characterization of chordal graphs due to [G. A. Dirac. On rigid circuit graphs. \emph{Abh. Math. Sem. Univ. Hamburg, 25, 71-76}, 1961] that use notions of a {\em simplicial vertex} and a {\em simplicial ordering}.
\end{abstract}
\section{Introduction}
Notations and definitions are as per \cite{golu,west}.  The {\em chordality} of a graph is the size of its longest induced (chordless) cycle. Acyclic graphs are assumed to be of  chordality $0$. Low chordality graphs are known to admit rich combinatorial structure and efficient algorithms.  In particular, chordal graphs \cite{golu}, which are graphs of chordality at most 3, have a structural characterization due to Dirac \cite{dirac} that says: a graph is chordal if and only if every minimal vertex separator is a clique.  A vertex is {\em simplicial} if its neighbourhood induces a clique.  Every chordal graph has a simplicial vertex and since chordality is a hereditary property, an ordering of vertices, referred to as {\em perfect vertex elimination ordering} or {\em simplicial ordering} can be obtained.  Also, a graph is chordal if and only if it has a perfect vertex elimination ordering \cite{FulkersonGross}.  Further, chordal graphs can be recognized in linear time using this ordering. For any integer $k\geq 3$, a graph is {\em $k$-chordal} if its chordality is at most $k$. Thus, chordal graphs are precisely the class of $3$-chordal graphs. Many interesting properties of $k$-chordal graphs have been studied in the literature \cite{chvatal,hayward,spinrad,small-chord,low-chord}.  In particular, a characterization of $k$-chordal graphs based on the existence of a {\em simplicial path} is known \cite{chvatal}.  Also, $k$-chordal graphs can be recognized in $O(n^{k})$ time \cite{hayward,spinrad} and this problem is known to be coNP-complete \cite{co-np}.  In this work, we explore structural characterizations of $k$-chordal graphs via minimal vertex separator and vertex ordering.  For a set $A \subsetneq V(G)$, the neighbourhood of $A$, denoted as $N_G(A)$ is $\{x \mid x \notin A, \exists y \in A$ such that $\{x,y\} \in E(G)\}$.  The closed neighbourhood of $A$ is $N_G[A]=A \cup N_G(A)$.  Let $P_{uv}$ denote a path between $u$ and $v$. The length of $P_{uv}$, denoted as $||P_{uv}||$, is the number of edges in $P_{uv}$.  Let $d_G(u,v)$ denote the length of a shortest path between $u$ and $v$ in $G$ and it is infinity if no such path exists.  For a graph $G$, the graph $G^k$ defined on the vertex set $V(G)$ with $E(G^k)=E(G) \cup \{\{u,v\} \mid u,v \in V(G)$ and $d_G(u,v) \leq k \}$ is referred to as the $k^{th}$ power of $G$.  For a vertex $v$ in $G$, $(G-v)^k$ refers to the $k^{th}$ power of $G$ induced by the vertex set $V(G) \setminus \{v\}$.  A set $A$ of vertices is referred to as a {\em connected non-dominating set} if $G[A]$ is connected and $N_G[A] \subsetneq V(G)$.  For every integer $k \geq 3$, a vertex $v$ is \emph{$k$-simplicial} in $G$ if
\begin{enumerate}
\item[](C1) $N_G(v)$ induces a clique in $(G-v)^{(k-2)}$
\item[](C2) for every non-adjacent pair $x,y$ in the neighbourhood of $v$, every chordless path between $x$ and $y$, with internal vertices excluding $v$ and its neighbourhood, has at most $k-2$ edges.
\end{enumerate}
Note that neither of the conditions (C1) or (C2) always implies the other.
\noindent For a graph $G$ on $n$ vertices, a vertex ordering $\pi = [v_1, v_2, \ldots , v_n]$ is a \emph{$k$-simplicial ordering} or \emph{$k$-perfect vertex elimination ordering}, if, for each $i$, $v_i$ is $k$-simplicial in $G[\{v_i,\ldots , v_n\}]$.  We show that a $k$-chordal graph has a $k$-simplicial vertex and as a consequence, we obtain a characterization of $k$-chordal graphs based on $k$-simplicial ordering.  In particular, we show the following theorem:
\begin{theorem}
\label{main-thm}
For a graph $G$ and an integer $k\geq 3$, the following statements are equivalent.\\
(i) $G$ is  $k$-chordal. \\
(ii) There exists a $k$-simplicial ordering of vertices in $G$.\\
(iii) Every minimal vertex separator $S$ in $G$ is such that, for all non-adjacent $x,y \in S$ and for any two distinct connected components $S_i$ and $S_j$ in $G \setminus S$, every pair $P^{S_i}_{xy}$ and $P^{S_j}_{xy}$ of induced paths between $x$ and $y$ with internal vertices from $S_i$ and $S_j$, respectively, satisfies $||P^{S_i}_{xy}||+||P^{S_j}_{xy}|| \leq k$.  
\end{theorem}
\begin{proof}
(i) if and only if (ii) follows from Theorem \ref{ordering-thm} and (ii) if and only if (iii) follows from Theorem \ref{k-chordal-char}.
\end{proof}
\noindent  We note that our study generalizes Dirac's structural results on chordal graphs, in particular, Lemma 4.2 and Theorem 4.1 in the book on perfect graphs by Golumbic \cite{golu}.
\section{Characterization of $k$-chordal graphs}
\begin{obs}
\label{l-simplicial-clique} 
A $k$-simplicial vertex is also $l$-simplicial, for every integer $l > k$.  Also, for each integer $k \geq 3$, every vertex in a complete graph is $k$-simplicial. 
\end{obs}
\begin{obs}
\label{condition2}
For an integer $k \geq 3$, every vertex in a $k$-chordal graph $G$ satisfies (C2). 
\end{obs}
\begin{obs}
\label{partition}
In a non-complete graph $G$, for every vertex $x$ which is non-adjacent to at least one other vertex, there exists a connected non-dominating set containing $x$.
\end{obs}
\begin{proof}
As there exists a vertex $y$ that is not adjacent to $x$, the set $A = \{x\}$ induces a connected subgraph with $y \notin N_G[A]$.  
\end{proof}
\begin{lemma}
\label{b-struct}
Let $A$ be a maximal connected non-dominating set in a non-complete graph $G$.  Every vertex in $V(G) \setminus N_G[A]$ is adjacent to every vertex in $N_G(A)$.
\end{lemma}
\begin{proof}
If there exists non-adjacent vertices $x' \in N_G(A)$ and $y' \in V(G) \setminus N_G[A]$, then the set $A' = A \cup \{x'\}$ is a connected non-dominating set in $G$ contradicting the maximality of $A$.  
\end{proof}
\begin{lemma}\label{b-simplicial}
Let $A$ be a maximal connected non-dominating set in a non-complete graph $k$-chordal graph $G$, where $k \geq 3$ is an integer.  There exists a vertex in $V(G) \setminus N_G[A]$ that is $k$-simplicial in $G$.
\end{lemma}
\begin{proof}
We prove the lemma by induction on $|V(G)|$.  It is easy to verify that the statement of the lemma is true when $|V(G)| \leq 3$. We assume the statement to be true for all $k$-chordal graphs having less than $n$ vertices, where $n \geq 4$. Consider a non-complete $k$-chordal graph $G$ on $n$ vertices.  Let $A$ be a maximal connected non-dominating set in $G$ and $B$ denote $V(G) \setminus N_G[A]$.  If $B$ is a clique, then by Observation \ref{l-simplicial-clique}, every vertex in $B$ in $k$-simplicial in $G[B]$. Otherwise, by induction hypothesis, there exists a vertex that is $k$-simplicial in $G[B]$.  Let $b \in B$ be $k$-simplicial in $G[B]$ and we now show that $b$ is $k$-simplicial in $G$ too.  Since $b \in B$, every vertex in $N_G(b)$ is either in $B$ or in $N_G(A)$. Let $x,y \in N_G(b)$. 
\\ \textbf{Case ($x,y \in B$)}: Since $b$ is $k$-simplicial in $G[B]$, there exists a path $P_{xy}$ in $G[B\setminus\{b\}]$ and thereby in $G[V(G)\setminus\{b\}]$ such that $||P_{xy}|| \leq (k-2)$. 
\\ \textbf{Case ($x \in N_G(A), y \in B$) or ($x \in B, y \in N_G(A)$)}: Since every vertex in $N_G(A)$ is adjacent with every vertex in $B$ by Lemma \ref{b-struct}, we can take the path $P_{xy}$ to be the edge $\{x,y\}$. 
\\\textbf{Case ($x,y \in N_G(A)$)}: If $\{x,y\} \in E(G)$, then the edge $\{x,y\}$ itself can be thought of as the path $P_{xy}$. Suppose $\{x,y\} \notin E(G)$. Since $A$ is connected and $x,y \in N_G(A)$, there exists a path between $x$ and $y$ whose every internal vertex is from $A$. Let $P_{xy}$ be the shortest of all such paths. Clearly, $P_{xy}$ is present in $G[V(G)\setminus \{b\}]$.  We claim that $||P_{xy}|| \leq k-2$. Suppose $||P_{xy}|| > k-2$. Then, $C = b P_{xy} b$ is an induced cycle of length at least $k+1$.  This contradicts the fact that $G$ is $k$-chordal. 
\end{proof}
\begin{lemma}
\label{simplicial}
For an integer $k \geq 3$, every $k$-chordal graph $G$ has a $k$-simplicial vertex. Moreover, if $G$ is not a complete graph, then it has two non-adjacent $k$-simplicial vertices. 
\end{lemma}
\begin{proof}
If $G$ is a complete graph, then by Observation \ref{l-simplicial-clique}, every vertex in $G$ is $k$-simplicial. Suppose $G$ is not a complete graph.  Let $A$ be a maximal connected non-dominating set in $G$ and $B$ denote $V(G) \setminus N_G[A]$.  The sets $A$ and $B$ are well-defined by Observation \ref{partition}.  By Lemma \ref{b-simplicial}, there exists a vertex $u \in B$ that is $k$-simplicial in $G$.  Now, let $A'$ be a maximal connected non-dominating set in $G$ containing $u$ and $B'$ denote $V(G) \setminus N_G[A']$.  The sets $A'$ and $B'$ are well-defined by Observation \ref{partition} as there exists a vertex $w \in A$ that is not adjacent to $u$.  By Lemma \ref{b-simplicial}, there exists a vertex $v \in B'$ that is $k$-simplicial in $G$.  Thus, $u$ and $v$ are two non-adjacent $k$-simplicial vertices in $G$. 
\end{proof}
\noindent As $k$-chordality is a hereditary property, we obtain the following characterization of $k$-chordal graphs. 
\begin{theorem}
\label{ordering-thm}
For an integer $k \geq 3$, a graph $G$ is $k$-chordal if and only if $G$ has a $k$-simplicial ordering. 
\end{theorem}
\begin{proof} 
$(\Rightarrow)$
As every induced subgraph of a $k$-chordal graph is $k$-chordal, the proof follows from Lemma \ref{simplicial}.   
\\
\noindent$(\Leftarrow)$Consider a $k$-simplicial ordering $\pi = [v_1,v_2,\cdots,v_n]$ of $G$. If $G$ is not $k$-chordal then there exists an induced cycle $C=(v_i, v_j, \ldots, v_l, v_i)$ of length at least $(k+1)$ with $v_i$ being the minimum labelled vertex in $C$ as per $\pi$. Clearly, the $2$ neighbours $v_j$ and $v_l$ of $v_i$ in $C$ satisfy $j,l > i$ by the choice of $v_i$. As $v_i$ is $k$-simplicial, there are no chordless paths between $v_j$ and $v_l$ of length greater than $(k-2)$ with internal vertices from $V(G) \setminus (\{v_1, \ldots, v_{i-1}\} \cup N_G[v_i])$ in $G[\{v_i,v_{i+1},\ldots ,v_n\}]$.  However, the path $P=v_j, \ldots, v_l$ in $C$ is one such path of length at least $(k-1)$, contradicting that $v_i$ is $k$-simplicial in $G[v_i, \ldots v_n]$. 
\end{proof}
\noindent It is well-known that a graph is chordal if and only if every minimal vertex separator induces a clique.  In the subsequent discussion, we generalize this result and characterize $k$-chordal graphs based on their minimal vertex separators.  We use terms minimal vertex separators and minimal $(a,b)$-vertex separators interchangeably and the pair $(a,b)$ under consideration will be clear from the context.  
\begin{theorem}
\label{k-chordal-char}
Let $k\geq 3$ be an integer. A graph $G$ is $k$-chordal if and only if for all minimal vertex separators $S$ in $G$, for all non-adjacent $x,y \in S$ and for any two distinct connected components $S_i$ and $S_j$ in $G \setminus S$, every pair $P^{S_i}_{xy}$ and $P^{S_j}_{xy}$ of induced paths between $x$ and $y$ with internal vertices from $S_i$ and $S_j$, respectively, satisfies $||P^{S_i}_{xy}||+||P^{S_j}_{xy}|| \leq k$.
\end{theorem}
\begin{proof}
$(\Rightarrow)$ Assume on the contrary that there exists a minimal vertex separator $S$ in $G$ containing two non-adjacent vertices $x$ and $y$ such that there exists chordless paths $P^{S_i}_{xy}$ and $P^{S_j}_{xy}$ with $||P^{S_i}_{xy}||+||P^{S_j}_{xy}|| > k$.  Clearly, $V(P^{S_i}_{xy}) \cup V(P^{S_j}_{xy})$ induces a cycle of length at least $k+1$ contradicting the fact that $G$ is $k$-chordal. \\ \noindent$(\Leftarrow)$ Suppose there exists an induced cycle $C$ of length at least $k+1$ in $G$.  Let $C=(x,a,y_1,\ldots,y_{(l-2)}=b,x), l \geq 3$.   We observe that any minimal $(a,b)$-vertex separator $S$ in $G$ must contain $x$ and some $y_r=z, 1 \leq r \leq (l-3)$.  This shows that there exists chordless paths $P^{S_i}_{xz}$ and $P^{S_j}_{xz}$ such that $||P^{S_i}_{xz}||+||P^{S_j}_{xz}|| \geq k+1$ contradicting our hypothesis.  Therefore, as claimed, $G$ is $k$-chordal.    
\end{proof}  
\bibliographystyle{plain}
\bibliography{ref}
\end{document}